\documentclass[letterpaper, 10 pt, journal, twoside]{ieeetran} 
\usepackage{graphicx}
\usepackage{amssymb}
\usepackage{amsfonts}

\usepackage{amsthm}
\usepackage{amsmath} 
\usepackage{mathtools}
\usepackage{balance}
\usepackage{url}
\usepackage{algpseudocode}
\usepackage{subfig}
\usepackage[colorlinks=true, urlcolor=blue, citecolor=blue, linkcolor=blue]{hyperref}
\usepackage{tikz}
\usepackage{blkarray}
\usetikzlibrary{shapes,arrows}
\usepackage{mathrsfs} 

\newcommand{\Span}{\operatorname{span}}

\newcommand{\Spn}{\mathbb{S}^n_+}
\newcommand{\Dpn}{\mathbb{D}^n_+}
\newcommand{\Id}{\mathbf{I}}
\newcommand*{\QEDB}{\hfill\ensuremath{\diamond}}%

\newcommand{\tr}{^\top}

\newcommand{\He}{\operatorname{He}}
\newcommand{\Diag}{\operatorname{diag}}

\newcommand{\0}{{\bf 0}}
\newcommand{\1}{{\bf 1}}

\newcommand{\R}{{\mathbb{R}}}
\newcommand{\NN}{{\mathbb{N}}}

\newcommand{\source}{{THIS IS A PREPRINT VERSION. IF YOU FOUND THIS READING USEFUL FOR YOUR RESEARCH PLEASE CITE THE PUBLISHED VERSION DOI: \href{https://doi.org/10.1109/LCSYS.2019.2925524}{https://doi.org/10.1109/LCSYS.2019.2925524}}}

\newtheorem{theorem}{Theorem}

\newtheorem{corollary}{Corollary}

\newtheorem{example}{Example}

\newtheorem{problem}{Problem}
\newtheorem{proposition}{Proposition}

\usepackage{fancyhdr}
\newcommand\x{\times}

\usepackage[customcolors]{hf-tikz}

\usetikzlibrary{fit}

\tikzset{style green/.style={
    set fill color=green!50!lime!60,
    set border color=white,
  },
  style cyan/.style={
    set fill color=cyan!50!,
    set border color=white,
  },
  style orange/.style={
    set fill color=orange!80!red!60,
    set border color=white,
  },
  hor/.style={
    above left offset={-0.15,0.31},
    below right offset={0.15,-0.125},
    #1
  },
  ver/.style={
    above left offset={-0.1,0.3},
    below right offset={0.15,-0.15},
    #1
  }
}

\newcommand{\fd}[1]{\textcolor{black}{#1}}

\makeatletter

\def\ps@IEEEtitlepagestyle{}
\title{On the Design of Structured Stabilizers for LTI Systems}
\author{Francesco~Ferrante, \IEEEmembership{Member, IEEE} Fabrizio Dabbene, \IEEEmembership{Senior Member, IEEE},\\ and Chiara Ravazzi, \IEEEmembership{Member, IEEE}%
\thanks{Francesco Ferrante is with Univ. Grenoble Alpes, CNRS, GIPSA-lab, F-38000 Grenoble, France. Email: francesco.ferrante@univ-grenoble-alpes.fr.}
\thanks{Fabrizio Dabbene and Chiara Ravazzi are with the National Research Council of Italy, CNR-IEIIT, c/o Politecnico di Torino, Corso Duca degli Abruzzi 14, 10129 Torino, Italy.
Email: fabrizio.dabbene@ieiit.cnr.it, chiara.ravazzi@ieiit.cnr.it.}%
\thanks{\textcolor{blue}{This file contains fixes to three typos in the published version. The typos are in blue font and there is a footnote explaining those. \textbf{These typos are minor and do not affect any of the results in the paper.} Last update: \today.}} 
}
\begin{document}
\tikzstyle{line} = [draw, -latex']
\maketitle
\begin{abstract}
Designing a static state-feedback controller subject to structural constraint achieving asymptotic stability is a relevant problem with many applications, including network decentralized control,
coordinated control, and sparse feedback design.
Leveraging on the Projection Lemma, this work presents a new solution to a class of state-feedback control problems, in which the controller is constrained to belong to a given linear space.
We show through extensive discussion and numerical examples that our approach leads to several advantages with respect to existing methods: first, it is computationally efficient; second, 
it is less conservative than previous methods, since it 
relaxes the requirement of restricting the Lyapunov matrix to a block-diagonal form.
\end{abstract}
\section{Introduction}
\IEEEPARstart{M}{odern} control engineering problems involve large-scale systems, frequently composed by different subsystems communicating between each others through complex interconnections. These require the design of control laws which are i) computationally efficient and ii) decentralized/distributed, reducing costly communication overhead \cite{zecevic2010control,mihailosurvey2016}. 

The design of state feedback control strategies for such systems requires the introduction of information structured constraints. Besides fully decentralized control \cite{wang1973stabilization,bakule2008overview,dandreaTAC}, in which the feedback gain $K$ is required to obey to a block-diagonal structure, more complicated constraints arise when one is interested, e.g.,\ in interconnected and/or coordinated controllers, 
see~\cite{antonellisurvey2013,mihailosurvey2016} and references therein.

It should be remarked that adding structural constraints to the control design problem immediately destroys its nice properties, even in the simplest case of fully decentralized state-feedback design. Indeed, this problem becomes nonconvex, and it is known to be NP-hard when bounds on the gain are present \cite{BlondelNP}.
Some positive results exist in the cases when the controllers $K$ are constrained to  a set that is \textit{quadratically invariant} with respect to the given system. When this property holds, the authors of \cite{LallCDC,Lall2006} showed that the problem can be  reduced to a tractable convex optimization adopting an operator design approach, via Youla parametrization \cite{Youla1976}. However, the resulting problem becomes infinite-dimensional, and needs to be approximated, usually leading  to arbitrarily high order  controllers.

Due to the importance of the problem, 
the last thirty years have witnessed a large body of literature on the problem. The available approaches can be divided into two main classes.

A first line of research directly tackles the nonconvex problem by means of different optimization-based techniques:
augmented Lagrangian and alternating direction method of multipliers (ADMM) in \cite{mihailoTAC,mihailoDistributed}, sequential convex-programming in
\cite{mihailoSCP}.
These approaches have the advantage of being rather general, and to allow to consider also sparsity-promoting approaches. On the other hand, they suffer the classical drawbacks of nonconvex optimization: the returned solution may result to  sub-optimal solutions and have usually few a-priori convergence guarantees.
An interesting line of research is represented by the use of moment-based and sum-of-squares optimization methods. In \cite{sznaier2016}, a convergent sequence of tractable convex relaxations with optimality certificates is derived by applying a polynomial optimization approach.  The proposed solution 
exploits the derivations in \cite{lavaei2013}, which shows that optimal decentralized control problem can be reformulated as a rank constrained optimization. On the same direction, \cite{chesi2017} shows that the polynomial-based approximation are able to provide necessary and sufficient conditions, provided that the order of the polynomial  is chosen to be sufficiently high.  These methods provide solutions allowing to trade off between conservatism and computational burden.

A second approach is based on a classical linear matrix inequalities (LMI) formulation. These approaches seek for particular cases in which the Lyapunov LMI admits a solution. 
To the best of our knowledge, all these approaches stem from the 
key technical observation that adopting a block-diagonal Lyapunov matrix leads to a solution to the Lyapunov LMI that is compatible to the desired structure.
In particular, the literature then concentrates on specific cases for which this choice turns out to be nonconservative.
This is for instance the case of internally-positive systems, for which the closed-loop state space matrix needs to be Metzler. It is a well-known fact \cite{positivematrices} that the  choice of a diagonal Lyapunov matrix is not conservative in this case, since Hurwitz Metzler matrices are diagonally stable. This consideration is at the basis of the results in \cite{tanaka2011}. Similar considerations hold for triangular systems, see \cite{masonchordal2014}.
Similarly,  in \cite{rantzer2016} LTI systems with symmetric and Hurwitz state (hence positive definite) matrices are considered.

Interesting results have also been derived in the case of so-called network-decentralized control
\cite{blanchiniCDC,blanchini2015network}, i.e.\ class of systems formed by decoupled subsystems interconnected via a set of controllers that can use state-information exclusively from the nodes they interconnect. 
The authors are able to derive conditions under which  the adoption of a block-diagonal Lyapunov matrix leads to feasible designs. In particular, when the subsystems do not share common unstable eigenvalues this choice is shown to be nonconservative. In the general case, sufficient conditions for solvability are derived.

However, it should be remarked that the cases above represent more the exception than the rule: in general the choice of diagonal Lyapunov matrices turns out to be very conservative, and may lead to infeasible designs.

The present work represents an important step forward in the direction of reducing the conservativeness of Lyapunov-based LMI approaches.
By recurring to the Projection Lemma \cite{gahinet1994linear} we introduce a reformulation which allows to separate the design of the state-feedback gain from the Lyapunov matrix design, so allowing to remove the necessity of adopting a diagonal matrix. 
Despite its simplicity, this approach proves to be rather powerful, allowing to significantly enlarge the applicability of these design techniques.

The remainder of this paper is organized as follows. We conclude this section with common notations and preliminary definitions used in the paper. Section~\ref{sec:ProblemStatement} addresses the general problem of structured state-feedback controller and presents some applications. The main  theoretical contribution is summarized in Section~\ref{sec:MainResults}. Numerical results showing the advantages of the proposed approach are provided in Section~\ref{sec:Examples}. Finally, some concluding remarks and discussions on future developments are reported in Section~\ref{sec:Conclusion}.\\\\
\textbf{Notation}: Let $\NN,\R$ be the set of natural and real numbers, respectively. Given $n\in\NN$ we use the notation $[n]=\{1,\ldots, n\}$. The symbol $\R^{n\times m}$ represents the set of $n\times m$ real matrices. The symbols $\Spn$ and $\Dpn$ denote, respectively, the set of real $n\times n$ symmetric positive definite matrices and the set of diagonal positive definite matrices. The symbol $\mathcal{R}^n$
defines the set of $n\times n$ nonsingular matrices.
Let $A\in\R^{m\times n}$, we denote its transpose by $A\tr ,$ and, when $n=m$, we define $\He(A)=A+A\tr .$ 
The matrix $\mathrm{diag}(A_1,A_2,\ldots,A_N)$ represents the block-diagonal matrix with $A_1,A_2,\ldots,A_N$ as diagonal blocks. For a symmetric matrix $A$, negative definiteness is denoted by $A\prec 0$. The symbol $\bullet$ stands for symmetric block in symmetric matrices. The symbols $\circ $ and $\otimes $ are used for element-wise and Kronecker product between matrices, respectively.
\section{Structured control design}
\label{sec:ProblemStatement}
\subsection{Mathematical formulation}
Let us consider the following linear \fd{time-invariant} plant 
\begin{equation}\label{eq:LTI}
\dot{x}=Ax+Bu+w,
\end{equation}
where $x\in\R^{n}$ is the plant state, $u\in\R^{m}$ is the control input, $w\in\R^{n}$ is an exogenous, non-controllable signal affecting the system, and $A\in\R^{n\times n}, B\in\R^{n\times m}$ are constant matrices.
In this paper, we are interested in designing a stabilizing static state-feedback controller subject to linear constraints. 
 More precisely, given the linear space $\mathcal{S}\subset\R^{m\times n}$, we address the problem of finding $K\in\mathcal{S}$ such that the closed-loop system with $u=Kx$ is stable, which is equivalent to the following problem.
\begin{problem}
\label{problem:MainProblem}
Let $\mathcal{S}\subset\R^{m\times n}$ be a linear space, $A\in\R^{n\times n}$, and $B\in\R^{n\times m}$ be constant matrices. Find a static state-feedback gain $K$ such that $K\in\mathcal{S}$ and $A+BK$ is Hurwitz.
\end{problem}

\subsection{Applications}
\label{sec:Cases}
The structured feedback control design subject to linear constraints as formulated in Problem \ref{problem:MainProblem} is very relevant and the general framework applies to several applications. 
\subsubsection{Network decentralized control}
\label{sec:Decentralized}
In \fd{network} decentralized control (see \cite{blanchini2015network} and reference therein) a class of $N$ linear interconnected systems is considered:
$$
\dot{x}_i=A_ix_i+\sum_{j\in\mathcal{C}_i}B_{ij}u_j+E_iw,
$$
where $x_i\in\R^{n_i}$ is the state of the $i$-th subsystem with $i\in[N]$, $\mathcal{C}_i$ is the set of indexes of control input acting on the $i$-th subsystem, $u_j\in\R^{m_j}$. The overall dynamics can be written in matrix form as in \eqref{eq:LTI}, where $A$ is a block-diagonal and $B$ is a given sparse matrix with a specific zero-pattern constrained by sets $\{\mathcal{C}_i\}_{i\in[N]}$:
\begin{gather*}A=\mathrm{diag}(A_1,\ldots,A_N),\quad A_i\in\R^{n_i\times n_i},\\
B_{ij}=0,\quad \forall j\notin \mathcal{C}_i,
\end{gather*}
and $\sum_{i=1}^Nn_i=n$, $\sum_{i=1}^Nm_i=m$.
Finding a decentralized control corresponds to design a state-feedback control where each component  affecting a certain subset of systems has information coming from the state components associated to those systems only.
This leads to Problem \ref{problem:MainProblem}, 
 enforcing the feedback matrix to have the same zero-block of matrix $B\tr $.
\textcolor{blue}{More precisely, define\footnote{\textcolor{blue}{In the published version, $0$ and  $\1$ should be swapped. Moreover, the symbol $\1$ was not explicitly defined and the size of each $\mathcal{S}(B\tr)_{ij}$ was not specified.}} 
$$\mathcal{S}(B\tr)_{ij}\coloneqq\begin{cases}
0&\text{if}\ B_{ji}\neq 0\\
\1&\text{otherwise},
\end{cases}
$$
where $\1$ is the matrix with all entries equal to one and each $\mathcal{S}(B\tr)_{ij}$ has the same size of $B_{ji}\tr$.}
Then, the structural constraint in the problem in \cite{blanchini2015network} can be formulated as a gain-constrained stabilization problem for which
\begin{equation}
\mathcal{S}=\{K\in\R^{m\times n}:\ K\circ \mathcal{S}(B\tr)=0\}.
\label{eq:HadamardK}
\end{equation}
It can be easily shown that, for a given $B$, the set of matrices satisfying \eqref{eq:HadamardK} is a subspace of $\R^{m\times n}$. 


\subsubsection{Coordinated Control}
In \cite{rantzer2016} $N$ subsystems are considered:
$$
\dot{x}_i=A_ix_i+B_{i}v_i+w_i,
$$
where $x_i\in\R^{n_i}$ is the state of the $i$-th subsystem with $i\in[N]$, $u_i\in\R^{m_i}$ are the control inputs that have to coordinate in order to satisfy the following constraint
$$
\sum_{i\in[N]}v_i=0.
$$
Given the coordination constraint, the main goal is to design a static state feedback controller. We can write the problem as in \eqref{eq:LTI} where $A$ and $B$ are block-diagonal
\begin{gather*}A=\mathrm{diag}(A_1,\ldots, A_N),\quad A_i\in\R^{n_i\times n_i},\\
B=\mathrm{diag}(B_1,\ldots ,B_N),\quad B_i\in\R^{m_i\times n_i}.\end{gather*} and 
$$x=(x_1\tr ,\ldots,x_{N}\tr )\tr, u=(v_1\tr ,\ldots,v_{N}\tr )\tr.$$ 
Then, finding a coordinated control law leads to seek $K$ in 
{\small{\begin{align*}
\mathcal{S}\!=\scalebox{0.8}{$\!\left\{\!\!\!\begin{array}{l}K\in\R^{m\times n}:\\
K\!=\!\left[\begin{array}{cccc}
\left(1-\frac{1}{N}\right)K_1&-\frac{1}{N}K_2&\dots&-\frac{1}{N}K_N\\
-\frac{1}{N}K_1&\left(1-\frac{1}{N}\right)K_2&\dots&-\frac{1}{N}K_N\\
-\frac{1}{N}K_1& -\frac{1}{N}K_2&\dots&-\frac{1}{N}K_N\\
\vdots&\vdots&\vdots&\vdots\\
-\frac{1}{N}K_1&-\frac{1}{N}K_2&\dots&\left(1-\frac{1}{N}\right)K_{N} 
\end{array}
\right]\end{array}\!\!\!\right\}.$}
\end{align*}}}
It is trivial to see that $\mathcal{S}$ is a linear subspace of $\R^{m\times n}$, with $n=\sum_{i\in[N]}n_i$ and $m=\sum_{i\in[N-1]}m_i$. \medskip
\subsubsection{Sparse Feedback Control}
Another piece of literature (see \cite{Polyak2014} and reference therein) considers the problem of designing a linear state feedback minimizing the non-zero components of the control vectors.
This yields to seek control gains with minimum non-zero rows. 
More precisely, define
$$
\mathcal{S}_k=\{K\in\R^{m\times n}:\ \|K\|_{r,0}\leq k\}
$$
with $\|K\|_{r,0}=\sum_{i=1}^m(\max_{j\in[n]}|K_{ij}|)^0$ and the convention that $0^0=0$.
Finding a stabilizer $K$ with minimum number of non-zero rows is a combinatorial problem. It prescribes to solve the following problem
\begin{align*}
&\min k,\quad\text{s.t. }  A+BK \text{ is Hurwitz},\ K\in\mathcal{S}_k.
\end{align*}
It should be noticed that $\mathcal{S}_k$ is not a linear space but the union of linear spaces.

If, instead, the zeros pattern is fixed a-priori then the problem will reduce to constrain $K$ to a linear space. Similarly, in the so-called \textit{overlapping control} the subsystems might share some state variables and the gain matrix is required to assume the general form in Fig.~\ref{fig:overlapping}, see \cite{zecevic2010control}. Examples of such models can be found in \cite{Polyak2014}.
\begin{figure}[h]
\begin{equation*}\label{eq:appendrow}
 (a)\ \left[\begin{array}{cccc}
    0  & 0  & 0 & 0 \\
   \tikzmarkin[hor=style cyan]{el_1}  \x & \x  & \x& \x\tikzmarkend{el_1}  \\
    0   & 0   & 0& 0 \\
    0   & 0   & 0  & 0\\
    \tikzmarkin[hor=style cyan]{row} \x  &  \x  &  \x &  \x \tikzmarkend{row}\\
  \end{array}\right]\quad
 (b)\ \left[\begin{array}{ccccc}
  \tikzmarkin[hor=style cyan]{elAA}  \x  & \x & 0 & 0&0 \\
     \x & \x \tikzmarkend{elAA}  & 0& 0&0\\
    0   & \tikzmarkin[hor=style cyan]{elBB} \x  & \x\tikzmarkend{elBB}   & 0&0 \\
    0  & 0  & \tikzmarkin[hor=style cyan]{rowCCC} \x &  \x&  \x \\
    0  & 0  &\x&  \x &  \x \tikzmarkend{rowCCC}\\
  \end{array}\right]
 \end{equation*}
\begin{center}
\end{center}\caption{Sparse feedback control: $(a)$ row-sparse gain matrix structure $(b)$ overlapping gain matrix structure.}\label{fig:overlapping}
\end{figure}
\section{Main Results}
\label{sec:MainResults}
\subsection{Preliminary discussion}
It is well known from Lyapunov theory  \cite{boyd1994linear} that Problem~\ref{problem:MainProblem} is equivalent to solve the following bilinear matrix inequality problem  
\begin{equation}
\label{eq:Lyapunov0}
(A+BK)\tr P+P(A+BK)\prec\0,\quad\mbox{s.t.}\, K\in\mathcal{S}, P\in\Spn
\end{equation}
or, by taking the dual of \eqref{eq:LTI} as in \cite{ebihara2015springer}, equivalently\footnote{Considering the dual of \eqref{eq:LTI} enables to ease the reformulation of \eqref{eq:Lyapunov0} as a linear matrix inequality. This approach is used throughout the paper.} 
\begin{equation}
\label{eq:Lyapunov}
(A+BK) P+P(A+BK)\tr\prec\0\quad\mbox{s.t.}\, K\in\mathcal{S}, P\in\Spn.
\end{equation}
Finding the solution to a bilinear matrix inequality is rather challenging. However, following a common strategy pursued in the literature; see, e.g., \cite{boyd1994linear}, one can recast \eqref{eq:Lyapunov} into a linear matrix inequality by means of suitable change of variables. Following this approach, we provide the following sufficient condition for the solution to Problem~\ref{problem:MainProblem}. 
\begin{proposition}\label{prop:Lyapunov2}
Let $\mathcal{S}_Y\subseteq\R^{m\times n} ,\ \mathcal{S}_P\subseteq \Spn$ such that it holds the following implication
$P^{-1}\in\mathcal{S}_P, Y\in\mathcal{S}_Y\Longrightarrow YP^{-1}\in\mathcal{S}$.
If
\begin{equation}\begin{split}
\label{eq:Lyapunov2}
&\He(AP+BY)\prec\0\\
&\qquad\qquad\mbox{s.t.} \quad Y\in\mathcal{S}_Y,\ P^{-1}\in\mathcal{S}_P,\end{split}
\end{equation}
is feasible, then also
\eqref{eq:Lyapunov} is feasible and a stabilizing control law is given by $K=YP^{-1}$.
\end{proposition}

\begin{proof}
Defining $Y=KP$ with $Y\in\mathcal{S}_Y$ and $P^{-1}\in\mathcal{S}_P$, from \eqref{eq:Lyapunov} we get
$\He(AP+BY)\prec\0$.
Thus, since by construction $K=YP^{-1}\in\mathcal{S}$, $K$ solves Problem~\ref{problem:MainProblem}.
\end{proof}

Proposition \ref{prop:Lyapunov2} suggests a strategy for structured stabilization:
\begin{enumerate}
\item identify a couple of sets $\mathcal{S}_Y\subseteq \R^{m\times n}$ and $\mathcal{S}_P\subseteq\Spn$ for which $YP^{-1}\in\mathcal{S}$ for any $Y\in\mathcal{S}_Y$ and $P\in\mathcal{S}_P$;
\item solve the LMI in \eqref{eq:Lyapunov2};
\item choose  $K= YP^{-1}$.
\end{enumerate}
We retrieve now the frameworks proposed in Section~\ref{sec:Cases}.
\begin{example}[Network decentralized control] It can be easily verified that if
$\mathcal{S}$ is the set of matrices having the same block structure of $B$, 
$\mathcal{S}_P$ will coincide with the set of positive definite matrices with the same block-structure of matrix $A$ and 
$\mathcal{S}_Y=\mathcal{S}$. 
\end{example}
\begin{example}[Coordinated control] 
If we take $\mathcal{S}_P$ in the set of positive definite matrices with diagonal block structure and 
$\mathcal{S}_Y=\mathcal{S}$, then $K=YP\in\mathcal{S}$ for any $Y\in\mathcal{S}$ and $P\in\mathcal{S}_P$
\end{example}
\begin{example}[Sparse Feedback Control]
It should be noticed that in case of Sparse Feedback control for any $P\in\Spn$, we have
$\mathcal{S}_Y=\mathcal{S}$ is the set of matrices having at most $k$ non-zero rows.
\end{example}

It is worth remarking that  Proposition \ref{prop:Lyapunov2}, providing only a sufficient condition to solve \eqref{eq:Lyapunov2}, can introduce some conservatism leading to the introduction of some constraints on the matrix $P$, as in Example 1 and 2. In this paper, leveraging on the projection lemma \cite{gahinet1994linear} and through the introduction of additional variables, we overcome this drawback. The stability conditions resulting from this approach are generally known as extended \cite{pipeleers2009extended} or dilated \cite{ebihara2015springer} stability characterizations.
\subsection{Contribution}
As a first step, let us consider the following result, which provides a ``dilated'' version of \eqref{eq:Lyapunov}. The proof of the result can be easily obtained by following the general methodology illustrated in \cite{pipeleers2009extended, ebihara2015springer} and it is omitted here for the sake of brevity.
\medskip

\begin{theorem}
\label{theo:Extended}
Let $P\in\Spn$ and $K\in{\cal S}$. The following items are equivalent.
\begin{itemize}
\item [$(i)$] The inequality \eqref{eq:Lyapunov} is true;
\item [$(ii)$] There exists $X\in\R^{n\times n}$ such that
\begin{equation}
\label{eq:ExtendedLMI}
\begin{bmatrix}
0&P\\
\bullet&0
\end{bmatrix}+\He\left(\begin{bmatrix}
(A+BK)\\
-I
\end{bmatrix}\begin{bmatrix}
X&X
\end{bmatrix}\right)\prec \0.
\end{equation}
\end{itemize}\QEDB
\end{theorem}

Building on the above result, the following ``design oriented'' condition can be obtained. The proof of the result follows directly from Theorem~\ref{theo:Extended}, hence it is omitted here. 
\begin{corollary}
\label{corol:Projection}
Assume that there exist $X\in\R^{n\times n}, P\in\Spn$, and $R\in\R^{m\times n}$ such that
\begin{align}
\label{eq:ExtendedLMIDesign_1}
&\begin{bmatrix}
0&P\\
\bullet&0
\end{bmatrix}+\He\left(\begin{bmatrix}
AX+BR& AX+BR\\
-X&-X
\end{bmatrix}\right)\prec \0,\\
&RX^{-1}\in{\cal S}.
\label{eq:ExtendedLMIDesign_2}
\end{align}
Then, $K=RX^{-1}$ solves Problem~\ref{problem:MainProblem}.\QEDB
\end{corollary}

The main advantage offered by the above result is that for this equivalent formulation of \eqref{eq:Lyapunov}, in principle the Lyapunov matrix $P$ is not subject to any constraint.

Now we are in a position to state the main result of this paper, whose proof is reported in the Appendix.
\begin{theorem}
\label{theo:Main}
Let $\{S_1, S_2, \dots, S_k\}$ be a basis of $\mathcal{S}$, let
$$
L \coloneqq [\ S_1\ |\ S_2\ |\ \dots\ |\ S_k\ ],
$$
and define the \textit{structure set}
\begin{equation}
\label{eq:ScalQ}
\Upsilon\coloneqq \{Q\in\R^{n\times n}\colon \exists \Lambda\in\mathcal{R}^{k} \,\,\text{s.t.}\,\,L (\Id_{k}\otimes Q)=L (\Lambda\otimes\Id_n)\}.
\end{equation}
Assume that there exist $P\in\Spn$, $R\in\Span\{S_1, S_2,\dots, S_k\}$, and a nonsingular $X\in\R^{n\times n}$ such that:
\begin{subequations}
\label{eq:cond2}
\begin{align}
\label{eq:cond21}
&X\in\Upsilon,\\
\label{eq:cond22}
&\begin{bmatrix}
0&P\\\bullet&0
\end{bmatrix}+\He\left(\begin{bmatrix}
AX+BR& AX+BR\\
-X&-X
\end{bmatrix}\right)\prec \0.
\end{align}
\end{subequations}
Then $K=RX^{-1}$ solves Problem~\ref{problem:MainProblem}. \QEDB
\end{theorem}

Condition \eqref{eq:cond2} requires the design matrix $X$ to be in the structure set $\Upsilon$ and to satisfy an LMI. It should be noticed that, from definition of structure set in \eqref{eq:ScalQ}, condition \eqref{eq:cond21} reduces to a system of linear equations and thus the  complexity of the method is mainly dictated by the solution to the LMI \eqref{eq:cond22}.

Moreover, we emphasize that in many applications condition \eqref{eq:cond21} corresponds essentially to constrain the zero pattern of the design matrix $X$. In next section we show that  both network decentralized control and overlapping control fall in the proposed framework. More importantly. Theorem~\ref{theo:Main} allows us to extend the range of applicability of these previous results, by reducing conservatism.
\section{Examples}
\label{sec:Examples}
In this section, we showcase our methodology in some numerical examples\footnote{Code is available at \url{https://github.com/f-ferrante/CDC19SF.}}.
 \subsection{Network decentralized control}
In this section we focus on the decentralized control problem, introduced in \cite{blanchini2015network} and briefly recalled in Section~\ref{sec:Decentralized}. More precisely, we show in an example that the conditions derived in Theorem \ref{theo:Main} are less restrictive than those presented in Proposition \ref{prop:Lyapunov2}. \subsubsection{Example}\label{ex1:Decentralized}
We consider a case where the number of subsystems is $N=3$, each of dimension $n_1=n_2=n_3=1$. Let\footnote{\textcolor{blue}{In the published version, the matrix $A$ was wrongly typed in the text. The value $1$ should be replaced with $0.1$.}}
$$A= \left[\begin{array}{ccc}
0&    0   &  0\\
     0 &    \textcolor{blue}{0.1}   &  0\\
     0   &  0    & 0\end{array}\right],
\quad B= \left[\begin{array}{cc} 1 & 0\\ 0 & 1\\ 1 & 1\end{array}\right].
 $$
The topology of the corresponding decentralized control system is schematically represented in Fig.~\ref{fig:graph}. It is worthwhile to notice that in this case the \emph{external connectivity} assumption (see Definition 6 in \cite{blanchini2015network}) is not true. In other words, none of the nodes is directly connected with the external environment so that the two control inputs $u_1$ and $u_2$ needs to be ``shared" among the systems.
\begin{figure}[!h]
     \centering
     \begin{tikzpicture}
  [scale=.9,auto=left]
  \node[style={circle,fill=cyan!50}] (n4) at (10,8)  {1};
  \node[style={circle,fill=cyan!50}] (n5) at (8,9)  {2};
  \node [style={circle,fill=cyan!50}](n3) at (10,10)  {3};
    \draw[-] (n4) -- node[below right] {$u_1$} (n3) ;
     \draw[-] (n5) -- node[above left] {$u_2$} (n3) ;
\end{tikzpicture}
\caption{Topology of the decentralized control system considered in Example~\ref{ex1:Decentralized}.}
\label{fig:graph}
\end{figure}
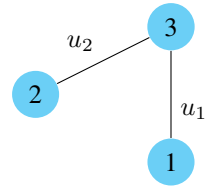
In this case, due to the structure of $B$, one has 
$$
\mathcal{S}(B\tr)=\left\{K\in\R^{2\times 3}\colon K=\begin{bmatrix}a& 0 &b\\0&c &d
\end{bmatrix}\vert\,(a, b, c, d)\in\R^4\right\}.
$$
We derive a structured stabilizer using Theorem~\ref{theo:Main}. To this end, as a first step, we give an explicit representation of the set $\Upsilon$ in the statement of Theorem~\ref{theo:Main}. In particular, by denoting $Q=(q_{i j})\in\R^{3\times 3}$ and $\Lambda=(\lambda_{i j})\in\mathcal{R}^{4}$, straightforward manipulations lead to  
$
\Upsilon=\{Q\in\R^{3\times 3}\colon \exists \Lambda\in\mathcal{R}^{4}\,\,\text{s.t.}\,\,\Pi(Q,\Lambda)=0\},
$
where $\Pi$ is defined in \eqref{eq:Pi} (at the top of the next page). 
\begin{figure*}[t!]{
\begin{equation}
\scalebox{0.8}{$\Pi(Q,\Lambda)=
\left[\begin{array}{cccccccccccc} q_{11}-\lambda_{11} & q_{12} & q_{13}-\lambda_{41} & -\lambda_{12} & 0 & -\lambda_{42} & -\lambda_{13} & 0 & -\lambda_{43} & q_{31}-\lambda_{14} & q_{32} & q_{33}-\lambda_{44}\\ 0 & -\lambda_{21} & -\lambda_{31} & q_{21} & q_{22}-\lambda_{22} & q_{23}-\lambda_{32} & q_{31} & q_{32}-\lambda_{23} & q_{33}-\lambda_{33} & 0 & -\lambda_{24} & -\lambda_{34} \end{array}
\right]
\label{eq:Pi}$}
\end{equation}}
\end{figure*}
Standard arguments enable to conclude that $\Pi(Q,\Lambda)=0$ for some nonsingular $\Lambda$ implies $q_{12}=q_{32}=q_{21}=q_{31}=0$, that is
$
Q=\left[\begin{smallmatrix}
q_{11} & 0 & q_{13}\\ 0 & q_{22} & q_{23}\\ 0 & 0 & q_{33}
\end{smallmatrix}\right].
$
Therefore, a solution to Problem~\ref{problem:MainProblem} can be obtained by solving \eqref{eq:ExtendedLMIDesign_1}
with $X=\left[\begin{smallmatrix}
x_{11} & 0 & x_{13}\\ 0 & x_{22} & x_{23}\\ 0 & 0 & x_{33}
\end{smallmatrix}\right]$ and
$R=\left[\begin{smallmatrix}
r_{11} & 0 & r_{13}\\ 0 & r_{22} & r_{23}
\end{smallmatrix}\right].$
In particular, by solving \eqref{eq:ExtendedLMIDesign_1} in Matlab\textsuperscript{\tiny\textregistered} using the YALMIP package \cite{lofberg2004yalmip} combined with the solver SDPT3~\cite{tutuncu2003solving}, one gets
$$
\begin{aligned}
&P=\left[\begin{array}{ccc} 14.06 & -2.755 & 0.6899\\ -2.755 & 9.04 & 7.419\\ 0.6899 & 7.419 & 20.76 \end{array}\right],\\
&R=\left[\begin{array}{ccc} -4.183 & 0 & -4.106\\ 0 & -4.95 & -4.663 \end{array}\right],\\
&X=\left[\begin{array}{ccc}9.328 & 0 & -1.598\\ 0 & 5.5 & 6.44\\ 0 & 0 & 12.52\end{array} \right],
\end{aligned}
$$ 
which leads to 
$$
K=\left[\begin{array}{ccc} -0.4484 & 0 & -0.3851\\ 0 & -0.9 & 0.0905 \end{array}\right]
$$
Using Proposition~\ref{prop:Lyapunov2}, Problem~\ref{problem:MainProblem} can be solved by enforcing $P^{-1}\in\Upsilon\cap\Spn$. However, it is worthwhile to observe that in this case, one has  $\Upsilon\cap\Spn=\Dpn$,
thereby implying $P\in\Dpn$. In other words, in this example the approach outlined in Proposition~\ref{prop:Lyapunov2} requires one to deal with a diagonal Lyapunov function and this can lead to a very conservative analysis. This drawback is largely mitigated by the proposed approach. 

According to \cite[Proposition 1]{blanchini2015network} a sufficient condition for the existence of a stabilizing decentralized control is to find a positive definite matrix $W$ sharing the same block-diagonal structure with $A$ and satisfying
\begin{equation}
\Gamma\coloneqq AW+WA\tr -2\gamma BB\tr \prec0.
 \label{eq:LMI_blanchini}
\end{equation}
Then, a decentralized controller is given by 
$K=-\gamma B^{\top }W^{-1}$.
In particular,  $W=\Diag(w_1, w_2, w_3)\in\mathbb{D}^3_+$ and \eqref{eq:LMI_blanchini} leads to 
$\det \Gamma={4\gamma^2 w_{2}}/{5}$. Since  $w_2>0$, $\Gamma$ cannot be negative definite and\footnote{\textcolor{blue}{In the published version, it is stated that Theorem 2 does not provide a viable approach for the solution to Problem 1. This is clearly a typo since Theorem 2 is used above to design the proposed gain $K$.}} \textcolor{blue}{\cite[Proposition 1]{blanchini2015network}} does not provide a viable approach for the solution to Problem~\ref{problem:MainProblem}. 
\subsection{Overlapping Control}
\label{ex:Overlapping}
We consider the \emph{control with overlapping information structure constraints} problem in \cite[Chapter 2]{zecevic2010control}  and we revisit \cite[Example 2.16]{zecevic2010control}. Let
$$
A=\left[\begin{array}{ccc}
1 & 4 & 0\\ 1 & 2 & 2\\ 0 & -2 & 3
\end{array}\right],
\quad B= \left[\begin{array}{cc} 1 & 0\\ 0 & 0\\ 0 & 1\end{array}\right],
$$
and 
$$
\mathcal{S}=\left\{K\in\R^{2\times 3}\colon K=\begin{bmatrix}a& b &0\\0&c &d
\end{bmatrix}\vert\,(a, b, c, d)\in\R^4\right\}.
$$
The setup analyzed in this example is schematically represented in Fig.~\ref{fig:OverlappinGraph}.

As observed in \cite{zecevic2010control}, the approach outlined in Proposition~\ref{prop:Lyapunov2} fails to determine a feasible feedback gain. 
In particular, our example falls in the so-called Type II problem defined in
\cite[Fig.\ 2.6]{zecevic2010control}, since the control inputs do not directly affect the ``shared'' state $x_2$. This is recognized to be a critical situation,
considerably more challenging than the Type I overlapping control, in which the shared state is directly influenced by the control inputs. In particular, in this case
the traditional approach based on the concept of expansion and the Inclusion Principle introduced in \cite{ikeda1984} does not apply, and in \cite{zecevic2010control} a two-steps procedure based on the introduction of \textit{low-rank centralized corrections} is proposed. This solution however violates by construction the structural constraints (even if with a low-impact term).

Now we show how our approach hinging upon Theorem~\ref{theo:Extended} can be successfully adopted in this case to solve Problem~\ref{problem:MainProblem}.
\begin{figure}[!h]
 \centering
\begin{tikzpicture}[->,>=stealth',auto,node distance=2cm,
  thick,main node/.style={circle,draw,font=\sffamily\Large\bfseries}, scale=1]
  \node[main node] (1) {$x_1$};
  \node[main node] (2) [right of=1] {$x_2$};
  \node[main node] (3) [right of=2] {$x_3$};
  \path[every node/.style={font=\sffamily\small}]
    (1) edge [bend left] node [right] {} (2)
    (2) edge  [bend left] node [right] {} (1)
    (3) edge [bend left] node [right] {} (2) 
    (2) edge [bend left] node [left] {} (3) ;
     \node[style={circle,fill=cyan!50}] (n4) at (-1.4,0){$u_1$};
         \node[style={circle,fill=cyan!50}] (n5) at (5.5,0){$u_2$};
         
\path [line] (n4) -- node [text width=2.5cm,midway,above ] {} (1);
\path [line] (n5) -- node [text width=2.5cm,midway,above ] {} (3);
\end{tikzpicture}
\caption{Topology of the overlapping control problem in Example~\ref{ex:Overlapping}.}
\label{fig:OverlappinGraph}
\end{figure}
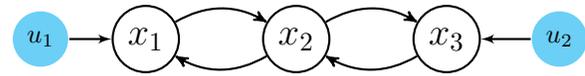
   
As a first step, notice that by following the same rationals as in Example~\ref{ex1:Decentralized}, it can be readily shown that the set $\Upsilon$ in this case reads as
$$
\Upsilon=\left\{Q\in\mathcal{R}^{3}:
Q=\begin{bmatrix}
q_{11} & q_{12} & 0\\ 0 & q_{22} & 0\\ 0 & q_{32} & q_{33}
\end{bmatrix}, q_{ij}\in\R\right\}.
$$
As such, also in this case, Problem~\ref{problem:MainProblem} can be solved by solving \eqref{eq:ExtendedLMIDesign_1}
with 
$$
\begin{aligned}
R&=\begin{bmatrix}
r_{11} & r_{12} & 0\\ 0 & r_{22} & r_{23}
\end{bmatrix},\qquad
X&=\begin{bmatrix}
x_{11} & x_{12} & 0\\ 0 & x_{22} & 0\\ 0 & x_{23} & x_{33}
\end{bmatrix}
\end{aligned} 
$$
In particular, by solving \eqref{eq:ExtendedLMIDesign_1} in Matlab\textsuperscript{\tiny\textregistered} using the YALMIP package \cite{lofberg2004yalmip} combined with the solver SDPT3~\cite{tutuncu2003solving}, one gets
$$
\begin{aligned}
&P=\left[\begin{array}{ccc} 15.6 & -4.453 & 1.708\\ -4.453 & 9.948 & -9.267\\ 1.708 & -9.267 & 22.14 \end{array}\right],\\
&R=\left[\begin{array}{ccc} -15.22 & -13.79 & 0\\ 0 & 24.65 & -24.86 \end{array}\right],\\
&X=\left[\begin{array}{ccc} 3.973 & -2.456 & 0\\ 0 & 4.038 & 0\\ 0 & -5.053 & 2.787\end{array} \right],
\end{aligned} 
$$ 
from which
$
K=\left[\begin{array}{ccc}  -3.831 & -5.744 &0\\ 0 & -5.059 & -8.922\end{array}\right].
$

\section{Concluding remarks}
\label{sec:Conclusion}
In this paper we developed a new method for static state feedback design with structural constraints. Compared to the state of the art, the main advantage of the proposed approach is its ability to eliminate the rigid block-diagonal structural constraints on the Lyapunov matrix, as shown in many examples, by still being computationally efficient. 
It is worth remarking that this work concentrated on the pure state feedback problem, without considering the minimization of $\mathcal{H}_2$ or $\mathcal{H}_\infty$ costs. This choice was done on purpose, since we feel that this simplest approach already captures the main advantage.
Indeed, extending the presented solutions to the optimal designs would simply require to substitute the Lyapunov
inequality with the LMIs arising from the KYP lemma, and the main ideas at the basis of our approach would also apply to that framework.
Moreover, the proposed design method can be also extended to include nonlinearities.

\section{Acknowledgments}
The authors would like to thank Dimitri Peaucelle for his valuable comments and suggestions.

\appendix
\begin{proof}[Proof of Theorem~\ref{theo:Main}]
Using inequality \eqref{eq:cond22} and applying Corollary~\ref{corol:Projection} we deduce that $K=RX^{-1}$ is a stabilizing gain. To conclude the proof, we only need to show that $K\in\mathcal{S}$. To this end, first we show that \eqref{eq:cond21} implies that $X^{-1}\in\Upsilon$. \textcolor{blue}{Pick $X\in\Upsilon$ nonsingular}. Then, there exists $\Omega\in\mathbb{R}^{k\times k}$ \textcolor{blue}{nonsingular}\footnote{\textcolor{blue}{From the definition of $\Upsilon$, $\Omega$ can be always selected to be nonnsingular and from the statement of Theorem~\ref{theo:Main} $X$ is nonsingular.}} such that
$
L(I_k\otimes X)=L(\Omega\otimes I_n).
$
Multiplying both members by $\Omega^{-1}\otimes X^{-1}$ we get
$
L(\Omega^{-1}\otimes I_n)=L(I_k\otimes X^{-1})
$
from which we deduce that $X^{-1}\in\Upsilon$. Since $X^{-1}\in\Upsilon$, there exists $\Lambda=(\lambda_{ij})\in\mathcal{R}^{k}$ such that
\begin{equation}
\begin{aligned}
&\left[\begin{array}{c|c|c|c} S_1X^{-1}&S_2X^{-1}&\dots &S_k X^{-1}\end{array}\right]=\\
&\left[\begin{array}{c|c|c|c} \sum_{i=1}^k\lambda_{i 1}S_i& \sum_{i=1}^k\lambda_{i 2}S_i&\dots&\sum_{i=1}^k\lambda_{i k}S_i
\end{array}\right].
\end{aligned}
\label{eq:span_proof}
\end{equation}
Moreover, since $R\in\Span\{S_1, S_2,\dots, S_k\}$, there exist $\alpha_1, \alpha_2, \dots, \alpha_k$ such that 
$
R=\sum_{\textcolor{blue}{j}=1}^k \alpha_{\textcolor{blue}{j}}S_{\textcolor{blue}{j}}
$, which yields
$
RX^{-1}=\sum_{\textcolor{blue}{j}=1}^k \alpha_{\textcolor{blue}{j}} S_{\textcolor{blue}{j}} X^{-1}.
$
Thus, by using \eqref{eq:span_proof} one has
$$
K=RX^{-1}=\sum_{j=1}^k \alpha_j \sum_{i=1}^k\lambda_{i j}S_i =\sum_{i=1}^kc_iS_i
$$
with $c_i=\sum_{j=1}^k\lambda_{ij}\alpha_j,$ from which we conclude $K\in\mathcal{S}$.
\end{proof}
\bibliographystyle{IEEEtran}
\bibliography{biblio_frg,biblio}
\end{document}